\documentclass{article}%
\usepackage{amsmath}
\usepackage{amssymb}
\usepackage{amsfonts}
\usepackage{graphicx}
\usepackage[singlespacing]{setspace}
\usepackage[singlelinecheck=true]{caption}%
\providecommand{\U}[1]{\protect\rule{.1in}{.1in}}
\newtheorem{theorem}{Theorem}

\newtheorem{corollary}[theorem]{Corollary}

\newtheorem{definition}[theorem]{Definition}
\newtheorem{example}[theorem]{Example}

\newtheorem{lemma}[theorem]{Lemma}

\newtheorem{proposition}[theorem]{Proposition}
\newtheorem{remark}[theorem]{Remark}

\newenvironment{proof}[1][Proof]{\noindent\textbf{#1.} }{\ \rule{0.5em}{0.5em}}
\begin{document}

\title{Factorizing a Finite Group into Conjugates of a Subgroup}
\author{Martino Garonzi\\Department of Mathematics\\University of Padova\\Via Trieste 63\\35121 Padova\\Italy
\and Dan Levy\\The School of Computer Sciences \\The Academic College of Tel-Aviv-Yaffo \\2 Rabenu Yeruham St.\\Tel-Aviv 61083\\Israel}
\maketitle

\begin{abstract}
For every non-nilpotent finite group $G$, there exists at least one proper
subgroup $M$ such that $G$ is the setwise product of a finite number of
conjugates of $M$. We define $\gamma_{\text{cp}}\left(  G\right)  $ to be the
smallest number $k$ such that $G$ is a product, in some order, of $k$ pairwise
conjugated proper subgroups of $G$. We prove that if $G$ is non-solvable then
$\gamma_{\text{cp}}\left(  G\right)  \leq36$ while if $G$ is solvable then
$\gamma_{\text{cp}}\left(  G\right)  $ can attain any integer value bigger
than $2$, while, on the other hand, $\gamma_{\text{cp}}\left(  G\right)
\leq4\log_{2}\left\vert G\right\vert $.

\end{abstract}

\section{ Introduction}

In this paper we consider representations of a finite group\footnote{Unless
otherwise stated, all our groups are assumed to be finite.} as a product of
conjugates of a single proper subgroup. This problem belongs to the broader
class of covering problems. By a covering of a finite group $G$ we mean a
collection of proper subsets of $G$, whose union or setwise product is $G$.
The covering operation (union or product) is fixed from the start, and in the
case that the covering operation is setwise product there may be restrictions
on the ordering of the subsets and their repetitions. Questions of interest
besides the mere existence of coverings of a specified type, include the
possible sizes of the coverings, and in particular, exact values or bounds on
minimal sizes. Several problems of this type are considered in the literature:
Union coverings by (conjugacy classes of) proper subgroups (for Union
coverings see \cite{tom}, \cite{lucdet}, \cite{garluc}, for Normal Union
coverings see \cite{bps}, \cite{mb}, \cite{gamma}), product coverings by
conjugacy classes (\cite{AradHerzog}), factorizing groups as a product of two
subgroups (\cite{LiebeckPraegerSaxl}), and other problems.

\begin{definition}
Let $G$ be a group. A \textit{conjugate product covering of }$G$ is a sequence
$\left(  A_{1},...,A_{k}\right)  $ of $k\geq2$ proper subgroups of $G$ such
that any two of the $A_{i}$ are conjugate in $G$ and $G=$ $A_{1}\cdots A_{k}$.
\end{definition}

Since a group $G$ is nilpotent if and only if every maximal subgroup of $G$ is
normal, a conjugate product covering of $G$ exists if and only if $G$ is non-nilpotent.

\begin{definition}
Let $G$ be a finite group. Define $\gamma_{\text{cp}}\left(  G\right)  $ to be
the minimal integer $k$ such that $G$ is a product of $k$ conjugates of a
proper subgroup of $G$ if $G$ is non-nilpotent, and $\gamma_{\text{cp}}\left(
G\right)  =\infty$ if $G$ is nilpotent (as usual $n<\infty$ for any natural
number $n$, and $\infty\leq\infty$).
\end{definition}

We remark that Liebeck, Nikolov and Shalev (\cite{LiebeckNikolovShalev1}%
,\cite{LiebeckNikolovShalev2}) also consider conjugate product coverings,
however, their discussion is limited from the outset to finite simple groups,
and concentrates on bounding the size of specific coverings in terms of the
orders of both the group and the covering subgroup.

Note (Lemma \ref{Lem_G=AB} below) that $\gamma_{\text{cp}}\left(  G\right)
>2$ for any group $G$ . For non-solvable groups our main result is the
existence of a universal constant bound on $\gamma_{\text{cp}}$.

\begin{theorem}
\label{Th_NScase_gamma<=36}Let $G$ be a non-solvable group. Then
$\gamma_{\text{cp}}\left(  G\right)  \leq36$.
\end{theorem}

In fact, we believe that $36$ is not the best possible bound (see Remark
\ref{Rem_Improving36}). On the other hand, for solvable groups we have:

\begin{theorem}
\label{Th_gamma=kForAllk>=3}For any integer $n\geq3$ there \ exists a solvable
group $G$ such that $\gamma_{\text{cp}}\left(  G\right)  =n$.
\end{theorem}

\begin{theorem}
\label{Th_SolvableCaseUpperBound}Let $G$ be a finite solvable group. Then
$\gamma_{\text{cp}}\left(  G\right)  \leq4\log_{2}\left\vert G\right\vert $.
\end{theorem}

The rest of the paper is organized as follows. In Section \ref{Section_QMNNG}
we collect some general results about $\gamma_{\text{cp}}$, and identify a
class of groups which we term quotient minimal non-nilpotent groups, on which
$\gamma_{\text{cp}}$ is maximal in a sense to be made precise. In sections
\ref{Section_NonSolvable} and \ref{Section_Solvable},\ we apply these general
results to proving Theorem \ref{Th_NScase_gamma<=36} and Theorems
\ref{Th_gamma=kForAllk>=3} and \ref{Th_SolvableCaseUpperBound} respectively,
as well as additional results and examples.

\textbf{Notation}. We use fairly standard notation. In particular,
$\mathbb{N}$ and $\mathbb{N}_{0}$ denote the positive and the non-negative
integers respectively, $\wr$ stands for wreath product, $\rtimes$ denotes
semi-direct product, $\Phi\left(  G\right)  $ and $F\left(  G\right)  $ are
the Frattini and Fitting subgroups of $G$, $T^{m}=\underset{m\text{ direct
factors}}{\underbrace{T\times...\times T}}$, and for $x$ real, $\left\lceil
x\right\rceil $ is the smallest integer satisfying $\left\lceil x\right\rceil
\geq x$.

\section{Quotient Minimal non-Nilpotent Groups\label{Section_QMNNG}}

The following lemma is a basic well-known result.

\begin{lemma}
\label{Lem_G=AB}Suppose that $G=AB$ for some subgroups $A$ and $B$ then
$G=A^{g_{1}}B^{g_{2}}$ for any $g_{1},g_{2}\in G$. In particular
$\gamma_{\text{cp}}\left(  G\right)  >2$ for every group $G$.
\end{lemma}

The next lemma is an immediate useful consequence of Lemma \ref{Lem_G=AB}.

\begin{lemma}
\label{Lem_HA_1...A_k}Let $H,A_{1},...,A_{k}\leq G$.

1. If $HA_{k}=G$ and $H\subseteq A_{1}\cdots A_{k}$ then $A_{1}\cdots A_{k}=G$.

2. If each of $A_{1},...,A_{k}$ is conjugate to $A\leq G$, $HA=G$ and
$H\subseteq A_{1}\cdots A_{k}$ then $A_{1}\cdots A_{k}=G$.
\end{lemma}

\begin{proof}
For (1) we have $G=HA_{k}\subseteq\left(  A_{1}\cdots A_{k}\right)
A_{k}=A_{1}\cdots A_{k}$, and (2) follows from (1) and Lemma \ref{Lem_G=AB}.
\end{proof}

The following is a key property for evaluating $\gamma_{\text{cp}}$.

\begin{proposition}
\label{Prop_Lifting}Let $G$ be a finite group. Then $\gamma_{\text{cp}}\left(
G\right)  \leq\gamma_{\text{cp}}\left(  G/N\right)  $ for every
$N\trianglelefteq G$. We shall call this "the lifting property".
\end{proposition}

\begin{proof}
We can assume that $G/N$ is non-nilpotent and hence $G/N=\overline{A}%
_{1}\cdots\overline{A}_{k}$ where $k=\gamma_{\text{cp}}\left(  G/N\right)  $
and the $\overline{A}_{i}<G/N$ are pairwise conjugated. Using the
correspondence theorem one shows that $G=A_{1}\cdots A_{k}$ where $A_{i}<G$ is
the inverse image of $\overline{A}_{i}$, and the $A_{i}$ are pairwise conjugated.
\end{proof}

\begin{definition}
A group $G$ is called \textit{a quotient minimal non-nilpotent group
(qmnn-group)} if $G$ is non-nilpotent but $G/N$ is nilpotent whenever
$\left\{  1_{G}\right\}  \neq N\trianglelefteq G$.
\end{definition}

Due to the lifting property $\gamma_{\text{cp}}\left(  G\right)  $ attains
maximal integer values on qmnn-groups, and hence we study their structure. Let
$N\left(  G\right)  $ denote the nilpotent residual of $G$. By definition,
this is the unique normal subgroup of $G$ which satisfies: $G/N\left(
G\right)  $ is nilpotent and for every $N\trianglelefteq G$ such that $G/N$ is
nilpotent we have $N\left(  G\right)  \leq N$. Note that $N\left(  G\right)  $
is the intersection of all $N\trianglelefteq G$ such that $G/N$ is nilpotent.

\begin{lemma}
\label{Lem_MinimalQNonNil}Let $G$ be a qmnn-group. Then:

a. $N\left(  G\right)  $ is the unique minimal normal subgroup of $G$. In
particular, $N\left(  G\right)  \cong T^{m}$ where $T$ is simple and
$m\in\mathbb{N}$, and $G$ is solvable if and only if $N\left(  G\right)  $ is
elementary abelian.

b. $\Phi\left(  G\right)  =Z\left(  G\right)  $ $=1$.

c. $G$ has a faithful primitive action.
\end{lemma}

\begin{proof}
a. Let $N$ be a minimal normal subgroup of $G$. Then $N>1$ and $G/N$ is
nilpotent. Hence $N\left(  G\right)  \leq N$. But $N\left(  G\right)  >1$
since $G$ is non-nilpotent, so $N\left(  G\right)  =N$.

b. First suppose that $\Phi\left(  G\right)  >1$. Then $G/\Phi\left(
G\right)  $ is nilpotent, so (\cite{Gorenstein2} Corollary 5.1.2) $G$ is
nilpotent - a contradiction. Hence $\Phi\left(  G\right)  =1$. Next Suppose
that $Z\left(  G\right)  >1$. Then $N\left(  G\right)  \leq Z\left(  G\right)
$. Since $N\left(  G\right)  $ is the lowest term in the lower central series,
$\left[  G,N\left(  G\right)  \right]  =N\left(  G\right)  $. But $N\left(
G\right)  \leq Z\left(  G\right)  $ so this gives $\left[  G,N\left(
G\right)  \right]  =1$ and therefore $N\left(  G\right)  =1$ - a contradiction.

c. It is sufficient to prove that one of the maximal subgroups of $G$ is
core-free. Suppose not. Then by (a) $N\left(  G\right)  \leq Core_{G}\left(
M\right)  $ for every maximal subgroup $M$ of $G$. It follows that $N\left(
G\right)  \leq\Phi\left(  G\right)  $ in contradiction to (b).
\end{proof}

Now we exhibit a connection between $\gamma_{\text{cp}}\left(  G\right)  $ and
the ranks of the permutation representations of $G$. Let $G$ be a group and
let $M$ be a proper subgroup of $G$. Set $\Omega=\left\{  Mg|g\in G\right\}  $
(the set of right cosets of $M$ in $G$). Then $G$ acts transitively by right
multiplication on $\Omega$ and the point stabilizer of $M1\in\Omega$ is $M$.
The action of $G$ induces an action of $M$ on $\Omega$ whose orbits are in
bijection with double cosets of $M$, when we view a double coset of $M$ as a
collection of right cosets of $M$: $MxM=\left\{  M\left(  xm\right)  |m\in
M\right\}  $ where $x\in G$. The number of $M$-orbits is denoted $r$ (the rank
of $G$). Note that $r\geq2$ and that $G$ acts $2$-transitively on $\Omega$ if
and only if $r=2$.

\begin{proposition}
\label{Prop_gammaBoundedByRank}Let $G$ be a group and let $M$ be any
non-normal maximal subgroup of $G$. Then $\gamma_{\text{cp}}\left(  G\right)
\leq r+1$. Moreover, if $r=2$ then $\gamma_{\text{cp}}\left(  G\right)  =3$.
\end{proposition}

\begin{proof}
Since $M$ is non-normal there exists a conjugate $M_{1}$ of $M$ such that
$M_{1}\nsubseteq M$. Let $x\in M_{1}-M$. Then $\left\{  M\right\}  $ and $MxM$
are two distinct orbits of the action of $M$. Set $B:=M\cup MxM$. We have:%
\begin{equation}
B^{k+1}=B^{k}\cup\left(  MxM\right)  ^{k+1}\text{, }\forall k\in
\mathbb{N}\text{.} \tag{*}%
\end{equation}
and in particular, $B^{k}\subseteq B^{k+1}$. Equation (*) can be proven by
induction on $k\geq1$ using $M^{2}=M$ and the fact that the setwise product of
$G$ subsets is distributive over union. By finiteness of $G$ there exists a
positive integer $k_{0}$ such that $B^{k_{0}}=B^{k_{0}+1}$. Choose $k_{0}$
which is minimal with respect to this property. Then for every $1\leq i\leq
k_{0}-1$ we have $B^{i}\subset B^{i+1}$. Observe that $B^{i}$ is a union over
a family of $M$-orbits, since each double coset of $M$ is an $M$-orbit, and
product of double cosets is a union of double cosets. We can thus conclude
that if $B^{i}\subset B^{i+1}$, then $B^{i+1}$ contains more orbits of the
action of $M$ than $B^{i}$. Since $B^{i}\subset B^{i+1}$ for all $1\leq i\leq
k_{0}-1$, and $B$ contains two orbits, the number of orbits which are
contained in $B^{k_{0}}$ is at least $k_{0}+1$.

Next observe that $B^{k_{0}}=B^{k_{0}+1}$ implies $\left(  B^{k_{0}}\right)
^{2}=B^{k_{0}}$. Thus $B^{k_{0}}$ is a subgroup. Since $M<B$ and $M$ is
maximal, we get $B^{k_{0}}=G$. By our previous argument it follows that
$k_{0}+1\leq r$.

Finally, using again Equation (*), we have:%
\[
\left(  M\cup MxM\right)  ^{k_{0}}=M\cup MxM\cup\left(  MxM\right)  ^{2}%
\cup...\cup\left(  MxM\right)  ^{k_{0}}\text{.}%
\]
Since $\left(  MxM\right)  ^{i}=MM^{x^{-1}}M^{x^{-2}}\cdots M^{x^{-i}}x^{i}$,
we get:%
\begin{gather*}
G=\left(  M\cup MxM\right)  ^{k_{0}}\\
=M\cup MM^{x^{-1}}x\cup MM^{x^{-1}}M^{x^{-2}}x^{2}\cup...\cup MM^{x^{-1}%
}M^{x^{-2}}\cdots M^{x^{-k_{0}}}x^{k_{0}}\text{.}%
\end{gather*}
Recall that $x\in M_{1}$, and hence $x^{i}M_{1}=M_{1}$ for any integer $i$ and
we get :%
\begin{gather*}
G=GM_{1}\\
=\left(  M\cup MM^{x^{-1}}x\cup MM^{x^{-1}}M^{x^{-2}}x^{2}\cup...\cup
MM^{x^{-1}}M^{x^{-2}}\cdots M^{x^{-k_{0}}}x^{k_{0}}\right)  M_{1}\\
=MM_{1}\cup MM^{x^{-1}}M_{1}\cup MM^{x^{-1}}M^{x^{-2}}M_{1}\cup...\cup
MM^{x^{-1}}M^{x^{-2}}\cdots M^{x^{-k_{0}}}M_{1}\\
=MM^{x^{-1}}M^{x^{-2}}\cdots M^{x^{-k_{0}}}M_{1}\text{,}%
\end{gather*}
where the last step follows from the fact that a product of a sequence of
subgroups contains the product of every subsequence of the sequence. It
follows that $\gamma_{\text{cp}}\left(  G\right)  \leq k_{0}+2\leq r+1$. If
$r=2$ then $\gamma_{\text{cp}}\left(  G\right)  >2$ forces $\gamma_{\text{cp}%
}\left(  G\right)  =3$.
\end{proof}

\begin{remark}
The rank $r$ of the action of $G$ on $\Omega$ is given by $r=%
{\textstyle\sum\limits_{\theta\in Irr\left(  G\right)  }}
m_{\theta}^{2}$, where $m_{\theta}$ is the multiplicity of the irreducible
complex character $\theta$ in the permutation character associated with the
action (see \cite{Isaacs}, Corollary (5.16)).
\end{remark}

\section{$\gamma_{\text{cp}}\left(  G\right)  $ for non-solvable
$G$\label{Section_NonSolvable}}

As we shall see, if $G$ is non-solvable then $\gamma_{\text{cp}}\left(
G\right)  $ is controlled by $\gamma_{\text{cp}}$ of non-solvable qmnn-groups.
Hence we consider the following setting.

\begin{description}
\item[Minimal Non-Solvable Setting] 

\item[1.] $G$ is a non-solvable group with a unique minimal normal subgroup
$N=soc\left(  G\right)  =T^{m}$, where $T$ is simple non-abelian and $m$ a
positive integer.

\item[2.] $X:=N_{G}\left(  T_{1}\right)  /C_{G}\left(  T_{1}\right)  $ where
$T_{1}$ is the first component of $T^{m}$.
\end{description}

Assuming the above setting, $X$ is an almost simple group with $soc\left(
X\right)  \cong T$ (for convenience we will set $T:=soc\left(  X\right)  $).
Furthermore (see Remark 1.1.40.13 of \cite{BEcl}), there is an embedding of
$G$ into $X\wr K=X^{m}\rtimes K$ where the action of $K$ as a transitive
permutation group on the components of $X^{m}$, is determined by the
permutation action of $G$ on the components of $N=T^{m}$. The embedding of $G$
into $X^{m}\rtimes K$ satisfies $GX^{m}=X^{m}\rtimes K$ and hence $K\cong
GX^{m}/X^{m}\cong G/G\cap X^{m}$. Note that $N=T^{m}\trianglelefteq
X^{m}\rtimes K$, but $G$ needs not contain $K$.

\begin{lemma}
\label{Lem_VT=X}Assume the minimal non-solvable setting. Let $V\leq X$ satisfy
$VT=X$. Set $M=V\cap T$. Then $G=N_{G}\left(  M^{m}\right)  N$.
\end{lemma}

\begin{proof}
Let $R=V\wr K\leq X\wr K$. Since $V$ normalizes $V\cap T=M$, and $K$
normalizes $M^{m}$, we have that $R$ normalizes $M^{m}$, whence $G\cap R\leq
N_{G}\left(  M^{m}\right)  $. Since $VT=X$ we get $RN=X\wr K$, and by
Dedkind's law,%
\[
G=G\cap\left(  RN\right)  =\left(  G\cap R\right)  N\leq N_{G}\left(
M^{m}\right)  N\text{.}%
\]
Since both $N_{G}\left(  M^{m}\right)  $ and $N$ are contained in $G$ we get
$G=N_{G}\left(  M^{m}\right)  N$.
\end{proof}

\begin{lemma}
\label{Lem_LiftingToNormalizers}Assume the minimal non-solvable setting.
Suppose that $U\leq X$ satisfies $UT=X$ and $\left(  U_{1}\cap T\right)
\cdots\left(  U_{h}\cap T\right)  =T$ where $U_{1},...,U_{h}$ are $h$
conjugates of $U$ in $X$. Then:%
\[
G=N_{G}\left(  \left(  U_{1}\cap T\right)  ^{m}\right)  \cdots N_{G}\left(
\left(  U_{h}\cap T\right)  ^{m}\right)  \text{.}%
\]
In particular, if $1<U\cap T<T$, then $G$ is a product of $h$ conjugates of a
proper subgroup of $G$ and $\gamma_{\text{cp}}\left(  G\right)  \leq h$.
\end{lemma}

\begin{proof}
Since $\left(  U_{i}\cap T\right)  ^{m}\leq N_{G}\left(  \left(  U_{i}\cap
T\right)  ^{m}\right)  $ for every $1\leq i\leq m$, we have
\begin{align*}
T^{m}  &  =\left(  \left(  U_{1}\cap T\right)  \cdots\left(  U_{h}\cap
T\right)  \right)  ^{m}=\left(  U_{1}\cap T\right)  ^{m}\cdots\left(
U_{h}\cap T\right)  ^{m}\\
&  \leq N_{G}\left(  \left(  U_{1}\cap T\right)  ^{m}\right)  \cdots
N_{G}\left(  \left(  U_{h}\cap T\right)  ^{m}\right)  \text{.}%
\end{align*}
Taking $V=U_{h}$ in Lemma \ref{Lem_VT=X} (by Lemma \ref{Lem_G=AB}, $U_{h}T=X$)
we conclude that $G=T^{m}N_{G}\left(  \left(  U_{h}\cap T\right)  ^{m}\right)
$. Now, $G=N_{G}\left(  \left(  U_{1}\cap T\right)  ^{m}\right)  \cdots
N_{G}\left(  \left(  U_{h}\cap T\right)  ^{m}\right)  $ follows from Lemma
\ref{Lem_HA_1...A_k}(1), with $H=T^{m}\leq G$, $k=h$ and $A_{i}=N_{G}\left(
\left(  U_{i}\cap T\right)  ^{m}\right)  \leq G$ for all $1\leq i\leq k$.

If $1<U\cap T<T$, then $1<\left(  U\cap T\right)  ^{m}<N=T^{m}$ and since $N$
is minimal normal in $G$, $N_{G}\left(  \left(  U\cap T\right)  ^{m}\right)  $
is a proper subgroup of $G$. Moreover, observe that for all $1\leq i\leq h$,
there exists $t_{i}\in T$ such that $U_{i}=U^{t_{i}}$. This follows from the
fact that $U_{i}$ is conjugate to $U$ in $X$. Hence there exists $x_{i}\in X$
such that $U_{i}=U^{x_{i}}$. However $X=UT$ so $x_{i}=u_{i}t_{i}$ with
$u_{i}\in U$ and $t_{i}\in T$ and hence $U_{i}=U^{x_{i}}=U^{u_{i}t_{i}%
}=U^{t_{i}}$. Furthermore, for all $1\leq i\leq h$, $U_{i}\cap T=U^{t_{i}}\cap
T=\left(  U\cap T\right)  ^{t_{i}}$. Since $T^{m}\leq G$, we can deduce that
for all $1\leq i,j\leq h$, $\left(  U_{i}\cap T\right)  ^{m}$ and $\left(
U_{j}\cap T\right)  ^{m}$ are conjugate in $G$. Finally, since normalizers of
conjugate subgroups are conjugate to each other, $N_{G}\left(  \left(
U_{i}\cap T\right)  ^{m}\right)  $ and $N_{G}\left(  \left(  U_{j}\cap
T\right)  ^{m}\right)  $ are conjugate in $G$, for every $i,j\in
\{1,\ldots,h\}$. This proves that $\gamma_{\text{cp}}\left(  G\right)  \leq h$.
\end{proof}

\begin{corollary}
\label{Coro_X/TForAlternating}Assume the minimal non-solvable setting with
$T\cong A_{n}$, $n\geq5$. Then $\gamma_{\text{cp}}\left(  G\right)  =3$.
\end{corollary}

\begin{proof}
For $n\neq6$, we have either $X\cong A_{n}$ or $X\cong S_{n}$. Now $T$ acts
$2$-transitively on $\left\{  1,...,n\right\}  $ with a point stabilizer which
is isomorphic to $A_{n-1}$. By Proposition \ref{Prop_gammaBoundedByRank},
$A_{n}$ is a product of three suitable conjugates of $A_{n-1}$. For $X\cong
A_{n}$ we can choose $U\cong A_{n-1}$ and for $X\cong S_{n}$ we can choose
$U\cong S_{n-1}$ so that in both cases $U$ satisfies all of the assumptions of
Lemma \ref{Lem_LiftingToNormalizers} with $h=3$, and hence $\gamma_{\text{cp}%
}\left(  G\right)  =3$. For $n=6$ we use the fact that $T\cong A_{6}$ has
another $2$-transitive action of degree $10$, whose point stabilizer is a
normalizer of a Sylow $3$-subgroup $P$ of $T$ (see, for example,
\cite{WebAtlas} permutation representations of $A_{6}$). Thus $T$ is a product
of three conjugates of $N_{T}\left(  P\right)  $. Since all of the normalizers
of Sylow $3$-subgroups of $T$ are conjugate in $T$ and $T\trianglelefteq X$,
we have, by the Frattini argument, that $X=N_{X}\left(  N_{T}\left(  P\right)
\right)  T$. Taking $U=N_{X}\left(  N_{T}\left(  P\right)  \right)  $ one
checks that $U$ satisfies all of the assumptions of Lemma
\ref{Lem_LiftingToNormalizers} with $h=3$, and therefore $\gamma_{\text{cp}%
}\left(  G\right)  =3$ also for $n=6$.
\end{proof}

\begin{corollary}
\label{Coro_X/TForSporadicAndTits}Assume the minimal non-solvable setting with
$T$ a sporadic simple group, or the Tit's group $^{2}F_{4}\left(  2\right)
^{\prime}$. Then $\gamma_{\text{cp}}\left(  G\right)  \leq36$.
\end{corollary}

\begin{proof}
Under our assumptions $\left\vert Aut\left(  T\right)  :T\right\vert \leq2$ so
$X$ is either $T$ or $Aut\left(  T\right)  $, where the second possibility
arises if $\left\vert Aut\left(  T\right)  :T\right\vert =2$. For each of the
27 possible $T$'s, and for each of the possible $X$ corresponding to a given
$T$, we wish to choose $U$ which satisfies the conditions of Lemma
\ref{Lem_LiftingToNormalizers} such that $U\cap T$ has the smallest rank\ with
respect to the action of $T$ on the coset space $\left\{  \left(  U\cap
T\right)  x|x\in T\right\}  $. By Proposition \ref{Prop_gammaBoundedByRank}
and Lemma \ref{Lem_LiftingToNormalizers}, $\gamma_{\text{cp}}\left(  G\right)
\leq$ $r+1$. For $X=T$ we choose $U$ to be a maximal subgroup of $T$ with
minimal rank. For $X=Aut\left(  T\right)  $ where $\left\vert Aut\left(
T\right)  :T\right\vert =2$, we choose $U$ to be a maximal subgroup of $X$
which is not contained in $T$, such that $U\cap T$ is maximal in $T$ and its
rank with respect to $T$ is minimal. Examining Table \ref{TblRanks} in the
appendix, which summarizes these choices, one finds that the largest bound,
$r+1=36$, is realized for $Aut\left(  O^{\prime}N\right)  $ with
$U=J_{1}\times2$.
\end{proof}

\begin{theorem}
\label{Th_gammaForUniqueNinNonAbelian}Assume the minimal non-solvable setting,
then $\gamma_{\text{cp}}\left(  G\right)  \leq36$.
\end{theorem}

\begin{proof}
We use the classification of finite simple non-abelian groups and split the
discussion according to the isomorphism type of $T$.

1. $T\cong A_{n}$, $n\geq5$. By Corollary \ref{Coro_X/TForAlternating}
$\gamma_{\text{cp}}\left(  G\right)  =3$.

2. $T$ is a simple group of Lie type of characteristic $p$. By Theorem D of
\cite{LiebeckPyber} we have that $T$ is a product of at most $25$ Sylow
$p$-subgroups (which are of course conjugate to each other by Sylow's
theorem). Let $P$ be a Sylow $p$-subgroup of $T$. By Frattini's argument
$X=N_{X}\left(  P\right)  T$. Now we can apply Lemma
\ref{Lem_LiftingToNormalizers} with $U=N_{X}\left(  P\right)  $. Note that
$U\cap T=N_{X}\left(  P\right)  \cap T=N_{T}\left(  P\right)  <T$ since $T$ is
simple and clearly $\left\{  1_{T}\right\}  <P\leq U\cap T$. In particular, we
can assume that $h$ in Lemma \ref{Lem_LiftingToNormalizers} satisfies
$h\leq25$. We deduce $\gamma_{\text{cp}}\left(  G\right)  \leq25$ whenever $T$
is a simple group of Lie type.

3. $T$ is one of the $26$ sporadic simple groups or $T$ is the Tit's group
$^{2}F_{4}\left(  2\right)  ^{\prime}$. By Corollary
\ref{Coro_X/TForSporadicAndTits} we have $\gamma_{\text{cp}}\left(  G\right)
\leq36$.

Thus $\gamma_{\text{cp}}\left(  G\right)  \leq36$.
\end{proof}

\begin{proof}
[Proof of Theorem \ref{Th_NScase_gamma<=36}]We proceed by induction on the
length $l\left(  G\right)  $ of a chief series of $G$ ($G$ can be any
non-solvable group). If $l\left(  G\right)  =1$ then $G$ is simple non-abelian
and $\gamma_{\text{cp}}\left(  G\right)  \leq36$ by Theorem
\ref{Th_gammaForUniqueNinNonAbelian}. If $l\left(  G\right)  >1$ there are two
possibilities to consider:

1. Either $G$ has an abelian minimal normal subgroup $N_{0}$, or all minimal
normal subgroups of $G$ are non-abelian and $G$ has at least two minimal
normal subgroups. In the first case set $N:=N_{0}$ and in the second case set
$N$ to any minimal normal subgroup of $G$. Then $G/N$ is non-solvable and
$l\left(  G/N\right)  <l\left(  G\right)  $, so, by induction, $\gamma
_{\text{cp}}\left(  G/N\right)  \leq36$ and therefore, by the lifting
property, $\gamma_{\text{cp}}\left(  G\right)  \leq36$.

2. $G$ has a unique minimal normal subgroup $N$ which is non-abelian. Then
$\gamma_{\text{cp}}\left(  G\right)  \leq36$ by Theorem
\ref{Th_gammaForUniqueNinNonAbelian}.
\end{proof}

\begin{remark}
\label{Rem_Improving36}We strongly suspect that the upper bound on
$\gamma_{\text{cp}}\left(  G\right)  $ where $G$ is non-solvable can be
significantly lowered. The "worst case" in the proof of Theorem
\ref{Th_gammaForUniqueNinNonAbelian} is associated with $Aut\left(  O^{\prime
}N\right)  $. After submitting the paper we have discovered a new method to
evaluate $\gamma_{\text{cp}}\left(  G\right)  $, where $G$ is a sporadic
almost simple group, which we believe ("work in progress") will eliminate this
case. Furthermore, in case (2) of Theorem \ref{Th_gammaForUniqueNinNonAbelian}
(groups of Lie type) we have taken a conservative approach in choosing to rely
on Theorem D of \cite{LiebeckPyber} which yields $\gamma_{\text{cp}}\left(
G\right)  \leq25$. Since \cite{LiebeckPyber} was published there appeared in
the literature claims for improving it. In \cite{BabaiNikolovPyber} it was
announced that every simple group of Lie type in characteristic $p$ is a
product of just five of its Sylow $p$-subgroups, although, as far as we know,
no complete proof has yet been published. A sketch of a proof for exceptional
Lie type groups appears in a survey by Pyber and Szabo (\cite{PyberSzabo}
Theorem 15). For classical Chevalley groups a better bound of four is claimed
by Smolensky, Sury and Vavilov in \cite{VavilovSmolenskySury}.
\end{remark}

\section{\bigskip$\gamma_{\text{cp}}\left(  G\right)  $ for solvable
$G$\label{Section_Solvable}}

If $G$ is solvable then it is clear that $\gamma_{\text{cp}}\left(  G\right)
$ is controlled by $\gamma_{\text{cp}}$ of solvable qmnn-groups. By Lemma
\ref{Lem_MinimalQNonNil}(c), these groups are primitive. Using known
properties of primitive solvable groups (Theorem (A15.6) of \cite{Doerk and
Hawkes}) we can assume the following setting in our discussion.

\begin{description}
\item[Minimal Solvable Setting] 

\item[1.] $G=V\rtimes K$, where $V$ is an elementary abelian group of order
$p^{n}$, $p$ a prime and $n$ a positive integer, $K$ is a non-trivial
irreducible nilpotent subgroup of $GL\left(  V\right)  \cong GL_{n}\left(
p\right)  $, with $k:=\left\vert K\right\vert $ not divisible by $p$, and
$\rtimes$ is the semi-direct product with respect to action of $K$ on $V$
obtained by restriction from the action of $GL\left(  V\right)  $ on $V$.

\item[2.] $V$ is the unique minimal normal subgroup of $G$, and all
complements to $V$ in $G$ are conjugate to $K$.
\end{description}

Note that the non-trivial action of $K$ on $V$ implies that $G$ is
non-nilpotent. When convenient we regard $V$ as a vector space of dimension
$n$ over the field $F_{p}$ of $p$ elements and use additive notation for $V$
and even a mixture of additive and multiplicative notation.

\begin{lemma}
Assume the minimal solvable setting. If $M\leq G$ is maximal then either
$M\cap V=1$ in which case $M$ is conjugate to $K$ or $V\leq M$, in which case
$M\trianglelefteq G$. In particular, $G$ is the product of $\gamma_{\text{cp}%
}\left(  G\right)  $ conjugates of $K$.
\end{lemma}

\begin{proof}
Suppose by contradiction that $1<M\cap V<V$. Then $V\nleq M$ and hence $G=MV$.
Now $M\cap V$ is normalized by $V$ since $V$ is abelian, and by $M$ since $M$
normalizes $V$ and itself. Hence $M\cap V\trianglelefteq MV=G$, contradicting
$1<M\cap V<V$, and the fact that $V$ is minimal normal in $G$. If $M\cap V=1$
then $M$ complements $V$ in $G$ and hence it is conjugate to $K$. If $M\cap
V=V$ then $V\leq M$, and then, since $G/V$ is nilpotent and $M$ maximal in
$G$, $M/V\trianglelefteq G/V$ and $M\trianglelefteq G$ by the correspondence theorem.
\end{proof}

\begin{lemma}
\label{Lem_vinKK^t}Assume the minimal solvable setting. For any $v\in V$ there
exists $t\in V$ such that $v\in KK^{t}$.
\end{lemma}

\begin{proof}
For any $x\in K$ set $C(x^{-1},V):=\left\{  x^{-1}vxv^{-1}=v^{x}-v|v\in
V\right\}  \subseteq V$. Note that $C(x^{-1},V)\leq V$ because
\[
v^{x}-v+u^{x}-u=\left(  v+u\right)  ^{x}-\left(  v+u\right)  \in
C(x^{-1},V)\text{, }\forall v,u\in V\text{.}%
\]
Since $V$ is abelian it is clear that $C(x^{-1},V)$ is normalized by $V$. We
now prove that if $x\in Z\left(  K\right)  $ then $C(x^{-1},V)$ is normalized
by $K$ as well:%
\[
\left(  v^{x}-v\right)  ^{y}=v^{xy}-v^{y}=v^{yx}-v^{y}=\left(  v^{y}\right)
^{x}-v^{y}\in C(x^{-1},V)\text{, }\forall v\in V,\forall y\in K\text{.}%
\]
Therefore, assuming $x\in Z\left(  K\right)  $ we get that $C(x^{-1}%
,V)\trianglelefteq G$. Since $C(x^{-1},V)\leq V$ and $V$ is minimal normal
this implies that either $C(x^{-1},V)=\left\{  0_{V}\right\}  $ or
$C(x^{-1},V)=V$. Suppose, in addition, that $x\neq1_{G}$ (since $K$ is
nilpotent such a choice of $x$ exists). Now $C(x^{-1},V)=\left\{
0_{V}\right\}  $ implies that $V$ centralizes $\left\langle x\right\rangle $
and since $x\in Z\left(  K\right)  $ it follows that $x\in Z\left(  G\right)
$ in contradiction to Lemma \ref{Lem_MinimalQNonNil}(b). Thus, if $x\neq1$ we
can conclude $C(x^{-1},V)=V$.

Let $v\in V$ be arbitrary. We wish to show that there exists $t\in V$ such
that $v\in KK^{t}$. Choose $1\neq x\in Z\left(  K\right)  $. Then
$C(x^{-1},V)=V$ and hence there exists $w\in V$ such that $v=w^{x}%
-w=x^{-1}wxw^{-1}=x^{-1}x^{w^{-1}}\in KK^{w^{-1}}$. Thus $t=w^{-1}$ satisfies
the claim.
\end{proof}

\begin{theorem}
\label{Th_SolvableBounds}Assume the minimal solvable setting. Then:%
\[
n\frac{\log_{2}p}{\log_{2}k}+1\leq\gamma_{\text{cp}}\left(  G\right)
\leq2n\left(  \log_{2}p+1\right)  \text{.}%
\]

\end{theorem}

\begin{proof}
1. Suppose that $G=K_{1}\cdots K_{h}$ where the $K_{i}$ are pairwise
conjugated subgroups of $G$. Then $p^{n}k=\left\vert G\right\vert
\leq\left\vert K_{1}\right\vert \cdots\left\vert K_{h}\right\vert =k^{h}$ and
the lower bound follows by taking logarithms.

2. Set $m:=\left\lceil \log_{2}p\right\rceil $. Let $\left\{  v_{1}%
,...,v_{n}\right\}  $ be a basis of the vector space $V$. If $i\in\left\{
1,...,n\right\}  $ and $s\in\left\{  1,...,p-1\right\}  $ then $s$ is a sum of
at most $m$ distinct powers of $2$ and hence $sv_{i}$ is a sum of at most $m$
vectors of the form $2^{j}v_{i}$ with $0\leq j\leq m-1$ (just write $s$ in
base $2$). By Lemma \ref{Lem_vinKK^t}, for every $i\in\left\{
1,...,n\right\}  $ and $j\in\left\{  0,...,m-1\right\}  $ there exist
$t_{ij}\in V$ such that $2^{j}v_{i}\in KK^{t_{ij}}$. Hence the product $%
{\textstyle\prod\limits_{j=0}^{m-1}}
KK^{t_{ij}}$ contains all elements of $V$ of the form $sv_{i}$, where $sv_{i}%
$, $0\leq s\leq p-1$, is written in multiplicative notation: For each
$j\in\left\{  0,...,m-1\right\}  $ we either pick $1$ from $KK^{t_{ij}}$ if
the $j$th bit of $s$ is zero or $2^{j}v_{i}$ if the $j$th bit of $s$ is $1$.
Hence:%
\[%
{\textstyle\prod\limits_{i=1}^{n}}
{\textstyle\prod\limits_{j=0}^{m-1}}
KK^{t_{ij}}=K%
{\textstyle\prod\limits_{i=1}^{n}}
{\textstyle\prod\limits_{j=0}^{m-1}}
KK^{t_{ij}}\supseteq K%
{\textstyle\prod\limits_{i=1}^{n}}
\left\langle v_{i}\right\rangle =KV=G\text{.}%
\]
This proves that $\gamma_{\text{cp}}\left(  G\right)  \leq2nm=2n\left(
\left\lceil \log_{2}p\right\rceil \right)  \leq2n\left(  \log_{2}p+1\right)  $.
\end{proof}

\begin{proof}
[Proof of Theorem \ref{Th_SolvableCaseUpperBound}]We can assume that $G$ is a
qmnn-group. Then, by Theorem \ref{Th_SolvableBounds} we have $\gamma
_{\text{cp}}\left(  G\right)  \leq2n\left(  \log_{2}p+1\right)  $. Since
$\left\vert G\right\vert =p^{n}k$, $4\log_{2}\left\vert G\right\vert
=4n\log_{2}p+4\log_{2}k\geq2n\left(  \log_{2}p+1\right)  $ and the claim follows.
\end{proof}

For the family of groups in the next example there is a true gap between the
lower and the upper bounds of Theorem \ref{Th_SolvableBounds}, and this may be
taken as a hint that a tighter upper bound exists.

\begin{example}
\label{Example_AGL1(F_p)}Assuming the minimal solvable setting take $n=1$ and
$p>2$, which gives $V\cong C_{p}$. Choose $K=Aut\left(  V\right)  \cong
C_{p-1}$. Then $G\cong$ $AGL_{1}\left(  F_{p}\right)  $ which acts
$2$-transitively on $F_{p}$ (See for instance \cite{DixonUndMortimer} Exercise
2.8.1 p.52). Hence $\gamma_{\text{cp}}\left(  G\right)  =3$ and the lower
bound is also $\left\lceil n\frac{\log_{2}p}{\log_{2}k}+1\right\rceil
=\left\lceil \frac{\log_{2}p}{\log_{2}\left(  p-1\right)  }+1\right\rceil =3$.
\end{example}

\subsection{Proof of Theorem \ref{Th_gamma=kForAllk>=3}}

\begin{proposition}
\label{Prop_D_2p}Let $p$ be an odd prime and let $G=D_{2p}$, the dihedral
group of order $2p$. Then $\gamma_{\text{cp}}\left(  G\right)  =\left\lceil
\log_{2}p\right\rceil +1$.
\end{proposition}

For proving Proposition \ref{Prop_D_2p} we need the following lemma.

\begin{lemma}
\label{Lem_X_n}Let $n\geq1$ be an integer. Set
\[
X_{n}:=\left\{
{\textstyle\sum\limits_{i=0}^{h}}
\left(  -1\right)  ^{i}2^{a_{i}}|0\leq h\leq n-1,a_{0}<...<a_{h}\leq
n-1,a_{i}\in\mathbb{N}_{0}\right\}  \text{.}%
\]
Then $X_{n}=\left\{  x\in\mathbb{Z}|-2^{n-1}+1\leq x\leq2^{n-1}\right\}
-\left\{  0\right\}  $, and:%
\[
\left\{  1,...,k\right\}  \subseteq X_{n}\operatorname{mod}\left(  k+1\right)
:=\left\{  x\operatorname{mod}\left(  k+1\right)  |x\in X_{n}\right\}  \text{,
}\forall1\leq k<2^{n}\text{.}%
\]

\end{lemma}

\begin{proof}
Set $Y_{n}:=\left\{  \left(  a_{0},...,a_{h}\right)  |0\leq h\leq
n-1,a_{0}<...<a_{h}\leq n-1,a_{i}\in\mathbb{N}_{0}\right\}  $. There is a
bijection between $Y_{n}$ and the set of non-empty subsets of $\left\{
0,...,n-1\right\}  $, and hence $\left\vert Y_{n}\right\vert =2^{n}-1$. We
prove by induction on $n\geq1$ that the natural mapping $Y_{n}\rightarrow
X_{n}$ is injective (it is clearly surjective). For $n=1$ there is nothing to
prove. Let $n>1$ and let $\left(  a_{0},a_{1},...,a_{h}\right)  \neq\left(
b_{0},b_{1},...,b_{h^{\prime}}\right)  $ be two elements of $Y_{n}$. Assume by
contradiction that
\[%
{\textstyle\sum\limits_{i=0}^{h}}
\left(  -1\right)  ^{i}2^{a_{i}}=%
{\textstyle\sum\limits_{i=0}^{h^{\prime}}}
\left(  -1\right)  ^{i}2^{b_{i}}\text{.}%
\]
If $a_{0}=b_{0}$ then $%
{\textstyle\sum\limits_{i=1}^{h}}
\left(  -1\right)  ^{i}2^{a_{i}}=%
{\textstyle\sum\limits_{i=1}^{h^{\prime}}}
\left(  -1\right)  ^{i}2^{b_{i}}$ and after canceling a common factor of $2$
on both sides we can apply the induction assumption and obtain a
contradiction. If $a_{0}>b_{0}$, then the left hand side is divisible by
$2^{b_{0}+1}$ while the right hand side is not - a contradiction. The case
$a_{0}<b_{0}$ is handled similarly. Thus $\left\vert X_{n}\right\vert =$
$2^{n}-1$. In order to complete the proof of the first claim of the lemma it
remains to check that $\min\left(  X_{n}\right)  =$ $-2^{n-1}+1$ (take $h=1,$
$a_{0}=1$ and $a_{1}=n-1$), that $\max\left(  X_{n}\right)  =2^{n-1}$ (take
$h=0,$ $a_{0}=n-1$), and that $0\notin X_{n}$ (Supposing $%
{\textstyle\sum\limits_{i=0}^{h}}
\left(  -1\right)  ^{i}2^{a_{i}}=0$, where $0\leq a_{0}<...<a_{h}$,
contradicts $%
{\textstyle\sum\limits_{i=0}^{h}}
\left(  -1\right)  ^{i}2^{a_{i}}\equiv2^{a_{0}}\left(  \operatorname{mod}%
2^{a_{0}+1}\right)  $). The second claim of the lemma is immediate if
$k\leq2^{n-1}$. If $2^{n-1}<k<2^{n}$ then any $x\in\left\{  2^{n-1}%
+1,...,k\right\}  $ is congruent, modulo $\left(  k+1\right)  $, to a number
in $\left\{  -2^{n-1}+1,-2^{n-1}+2,...,-1\right\}  $.
\end{proof}

\begin{proof}
[Proof of Proposition \ref{Prop_D_2p}]We use the familiar presentation of
dihedral groups, $D_{2p}=\left\langle v,b|v^{p}=b^{2}=1,bvb=v^{-1}%
\right\rangle $. Note that $G$ fits the minimal solvable setting with
$V=\left\langle v\right\rangle $ and $K=\left\langle b\right\rangle $. Set
$m:=\left\lceil \log_{2}p\right\rceil $. By Theorem \ref{Th_SolvableBounds}
$\gamma_{\text{cp}}\left(  G\right)  \geq m+1$ so it remains to prove
$\gamma_{\text{cp}}\left(  G\right)  \leq m+1$. For each $1\leq j\leq m$ set
$v_{j}:=v^{2^{j-1}}$, and $B:=\left\langle b\right\rangle ^{v_{1}}%
\cdots\left\langle b\right\rangle ^{v_{m}}\left\langle b\right\rangle $. We
shall prove that $V\subseteq B$ and then, by Lemma \ref{Lem_HA_1...A_k}, $B=G$
and $\gamma_{\text{cp}}\left(  G\right)  \leq m+1$ follows.

Observe that for any $1\leq i\leq m$ the defining relations of the $D_{2p}$
presentation imply $v_{i}^{-1}bv_{i}=$ $v_{i}^{-2}b$. Consequently $b^{v_{i}%
}b^{v_{j}}=v_{i}^{-2}v_{j}^{2}=\left(  v_{i}^{-1}v_{j}\right)  ^{2}$ for all
$1\leq i\neq j\leq m$. Thus, for any $1\leq t\leq m$, and $1\leq i_{1}%
<i_{2}<...<i_{t}\leq m$ we have
\begin{align*}
\left(  v_{i_{1}}^{-1}v_{i_{2}}v_{i_{3}}^{-1}v_{i_{4}}\cdots v_{i_{t-1}}%
^{-1}v_{i_{t}}\right)  ^{2}  &  \in\left\langle b\right\rangle ^{v_{1}}%
\cdots\left\langle b\right\rangle ^{v_{m}}\subseteq B\text{, if }t\text{ is
even}\\
\left(  v_{i_{1}}^{-1}v_{i_{2}}v_{i_{3}}^{-1}v_{i_{4}}\cdots v_{i_{t-2}}%
^{-1}v_{i_{t-1}}v_{i_{t}}^{-1}\right)  ^{2}  &  \in\left\langle b\right\rangle
^{v_{1}}\cdots\left\langle b\right\rangle ^{v_{m}}b\subseteq B\text{, if
}t\text{ is odd.}%
\end{align*}
Substituting $v_{j}:=v^{2^{j-1}}$, $1\leq j\leq m$, we conclude that every
element of $V$ of the form $v^{-2x}=\left(  v^{-2}\right)  ^{x}$ where $x\in
X_{m}$ ($X_{m}$ is defined in Lemma \ref{Lem_X_n}) is in $B$. Note that since
$p$ is odd $v^{-2}$ is a generator of $V$. By definition of $m$ we have
$p-1<2^{m}$, and hence, by Lemma \ref{Lem_X_n}, $\left\{  1,...,p-1\right\}
\subseteq X_{m}\operatorname{mod}p$. Since $v^{0}=1_{G}\in B$ as well, we get
$V\subseteq B$.
\end{proof}

\begin{proof}
[Proof of Theorem \ref{Th_gamma=kForAllk>=3}]By Bertrand's postulate, for
every integer $n\geq3$ there exists at least one odd prime $p$ such that
$2^{n-2}<p<2^{n-1}$. Hence $\left\lceil \log_{2}p\right\rceil =n-1$, and, by
Proposition \ref{Prop_D_2p}, $\gamma_{\text{cp}}\left(  D_{2p}\right)
=\left\lceil \log_{2}p\right\rceil +1=n$.
\end{proof}

\begin{remark}
Here are two additional relevant results, stated without proofs.

1. For a general dihedral group $D_{2n}$ with $n\geq2$ an arbitrary integer,
$\gamma_{\text{cp}}\left(  D_{2n}\right)  =\infty$ if $n$ is a power of $2$
and otherwise $\gamma_{\text{cp}}\left(  D_{2n}\right)  =\left\lceil \log
_{2}p\right\rceil +1$, where $p$ is the smallest odd prime divisor of $n$.

2. One can generalize the ideas behind the proof of Proposition
\ref{Prop_D_2p}, for the case that $G$ is a subgroup of $AGL_{1}\left(
F_{p^{n}}\right)  $, $p$ is an odd prime and $n$ is a positive integer, with
$V\cong\left(  F_{p^{n}},+,0\right)  $ and $K$ acts irreducibly by
multiplication on $V$ as a subgroup of $\left(  F_{p^{n}}^{\ast}%
,\cdot,1\right)  $. In particular, for $p=13$, $n=1$ and $K$ the order $4$
subgroup of $F_{13}^{\ast}$, one obtains $\gamma_{\text{cp}}\left(  G\right)
=4$. Note that the lower bound of Theorem \ref{Th_SolvableBounds} for this
case is $3$ (compare to Example \ref{Example_AGL1(F_p)}).
\end{remark}

\noindent{\huge Appendix} \label{appendix}

Table \ref{TblRanks} below presents a choice of $U\leq X$ for each almost
simple sporadic group $X$, such that $U$ satisfies the conditions of Lemma
\ref{Lem_LiftingToNormalizers} with as minimal as possible value of $r$. The
values of $h=r+1$ given here provide upper bounds on $\gamma_{\text{cp}%
}\left(  G\right)  $ in the proof of Corollary
\ref{Coro_X/TForSporadicAndTits}. The table is based on two sources:

1. Breuer and Lux (\cite{BreuerLux},\cite{Breuer}) have computed all
multiplicity free permutation characters of almost simple sporadic groups.
Note that $Aut\left(  O^{\prime}N\right)  $ has no suitable multiplicity free
permutation characters.

2. Using GAP's character library (\cite{GAP},\cite{GAPCTblLib1.2.1}) we have
been able to compute the permutation characters associated with maximal
subgroups of all almost simple sporadic groups beside the groups $B$ and $M$,
and $^{2}F_{4}\left(  2\right)  ^{\prime}$. In particular one can estimate $h$
for $Aut\left(  O^{\prime}N\right)  $ in this way. ~

\begin{center}
\bigskip%
\begin{table}[t] \centering
\begin{tabular}
[c]{|c|c|c||c|c|c||c|c|c|}\hline
$X$ & $U$ & $h$ & $X$ & $U$ & $h$ & $X$ & $U$ & $h$\\\hline
$M_{11}$ & $A_{6}.2_{3}$ & $3$ & $M_{24}$ & $M_{23}$ & $3$ & $HN.2$ & $4.HS.2$
& $10$\\\hline
$M_{12}$ & $M_{11}$ & $3$ & $M^{c}L$ & $U_{4}\left(  3\right)  $ & $4$ & $Ly$
& $G_{2}\left(  5\right)  $ & $6$\\\hline
$M_{12}.2$ & $L_{2}(11).2$ & $6$ & $M^{c}L.2$ & $U_{4}\left(  3\right)
.2_{3}$ & $4$ & $Th$ & $^{3}D_{4}(2).3$ & $12$\\\hline
$J_{1}$ & $L_{2}(11)$ & $6$ & $He$ & $S_{4}(4).2$ & $6$ & $Fi_{23}$ &
$2.Fi_{22}$ & $4$\\\hline
$M_{22}$ & $L_{3}(4)$ & $3$ & $He.2$ & $S_{4}(4).4$ & $6$ & $Co_{1}$ &
$Co_{2}$ & $5$\\\hline
$M_{22}.2$ & $L_{3}(4).2_{2}$ & $3$ & $Ru$ & ${}^{2}F_{4}\left(  2\right)
^{\prime}.2$ & $4$ & $J_{4}$ & $2^{11}:M_{24}$ & $8$\\\hline
$J_{2}$ & $U_{3}(3)$ & $4$ & $Suz$ & $G_{2}(4)$ & $4$ & $Fi_{24}^{\prime}$ &
$Fi_{23}$ & $4$\\\hline
$J_{2}.2$ & $U_{3}(3).2$ & $4$ & $Suz.2$ & $G_{2}(4).2$ & $4$ & $Fi_{24}%
^{\prime}.2$ & $Fi_{23}\times2$ & $4$\\\hline
$M_{23}$ & $M_{22}$ & $3$ & $O^{\prime}N$ & $L_{3}(7).2$ & $6$ & $B$ &
$2.^{2}E_{6}(2).2$ & $6$\\\hline
${}^{2}F_{4}\left(  2\right)  ^{\prime}$ & $L_{3}(3).2$ & $5$ & $O^{\prime
}N.2$ & $J_{1}\times2$ & $36$ & $M$ & $2.B$ & $10$\\\hline
${}^{2}F_{4}\left(  2\right)  ^{\prime}.2$ & $2.[2^{9}]:5:4$ & $6$ & $Co_{3}$
& $M^{c}L.2$ & $3$ &  &  & \\\hline
$HS$ & $U_{3}\left(  5\right)  .2$ & $3$ & $Co_{2}$ & $U_{6}(2).2$ & $4$ &  &
& \\\hline
$HS.2$ & $M_{22}.2$ & $4$ & $Fi_{22}$ & $2.U_{6}(2)$ & $4$ &  &  & \\\hline
$J_{3}$ & $L_{2}(16).2$ & $9$ & $Fi_{22}.2$ & $2.U_{6}(2).2$ & $4$ &  &  &
\\\hline
$J_{3}.2$ & $L_{2}(16).4$ & $9$ & $HN$ & $2.HS.2$ & $10$ &  &  & \\\hline
\end{tabular}
\caption{Subgroups of Almost Simple Sporadic Groups which provide the best upper bounds on $\gamma_{\text{cp}}$ via the rank argument}\label{TblRanks}%
\end{table}%

\bigskip

\bigskip

\bigskip

\bigskip

\bigskip

\bigskip

\bigskip

\bigskip
\end{center}

\end{document}